\documentclass[reqno]{amsart}
\usepackage{amssymb,amscd,verbatim,graphicx,mathtools}
\usepackage[mathscr]{euscript}
\newtheorem{thm}{Theorem}[section]

\newtheorem{lem}[thm]{Lemma}

\theoremstyle{definition}
\newtheorem{defn}[thm]{Definition}

\numberwithin{equation}{section}

\newcommand{\bb}[1]{{\mathbb{#1}}}
\newcommand{\mc}[1]{{\mathcal{#1}}}

\newcommand{\e}{\operatorname{e}}
\newcommand{\bm}{\operatorname{m}}

\newcommand{\tp}{\tilde p}
\newcommand{\tq}{\tilde q}
\newcommand{\tr}{\tilde r}
\newcommand{\ts}{\tilde s}
\newcommand{\tu}{\tilde u}
\newcommand{\tv}{\tilde v}
\newcommand{\tV}{\tilde V}

\newcommand{\iOmega}{(a,b)}

\begin{document}
\title{On Leighton's Comparison Theorem}%
\author{Ahmed Ghatasheh and Rudi Weikard}%
\address{A.G.: Department of Mathematics, University of Alabama at Birmingham, Birmingham, AL 35226-1170, USA}%
\email{ghatash@uab.edu}%
\address{R.W.: Department of Mathematics, University of Alabama at Birmingham, Birmingham, AL 35226-1170, USA}%
\email{rudi@math.uab.edu}%

\date{\today}%
\thanks{Name of \TeX{} file: \texttt{comp.tex}}%

\keywords{Comparison theorem; distributional potentials}

\begin{abstract}
We give a simple proof of a fairly flexible comparison theorem for equations of the type $-(p(u'+su))'+rp(u'+su)+qu=0$ on a finite interval where $1/p$, $r$, $s$, and $q$ are real and integrable.
Flexibility is provided by two functions which may be chosen freely (within limits) according to the situation at hand.
We illustrate this by presenting some examples and special cases which include Schr\"odinger equations with distributional potentials as well as Jacobi difference equations.
\end{abstract}

\maketitle

\section{Introduction}
In 1836 Sturm published his paper \cite{Sturm1836} containing the celebrated comparison theorem.
For it he studied two equations of the form $-(pu')'+qu=0$ to conclude something about the zeros of the solutions of one equation from the zeros of some solution of the other equation.
In fact, Sturm's theorem requires $p=\tp>0$, $\tq>q$ and the continuity of these coefficients.
Then, assuming that $-(p\tu')'+\tq\tu=0$ and $-(pu)'+qu=0$, that $a$ and $b$ are two consecutive zeros of $\tu$, and that $\tu$ and $u$ are positive in $(a,b)$, one obtains the contradiction
$$0<\int_a^b (\tq-q)\tu u=\int_a^b (up\tu'-pu'\tu)' = u p\tu'\bigg|^b_a  \leq 0.$$
It follows that, all else being the same, $u$ must have a zero in $(a,b)$.
The most prominent application of this result is in the oscillation theorem which compares the number of zeros of solutions of the equation $-(pu')'+(q-\lambda)u=0$ for different values of $\lambda$.

Only in 1909 Picone \cite{Picone1910} was able to weaken the condition $p=\tp$.
The key was the identity
$$\left(\frac{\tu}{u}(\tp\tu'u-\tu pu')\right)'=(\tp-p)\tu^{\prime2}+(\tq-q)\tu^2+\frac{p}{u^2}(\tu u'-\tu'u)^2$$
now known as Picone's identity.
Assuming $0<p\leq\tp$ and $q<\tq$ will yield a similar result as before by a similar argument.

Another essential improvement is due to Leighton \cite{MR0140759} who recognized that it was enough to require
$$\int_a^b [(p-\tp) \tu^{\prime2}+ (q-\tq)\tu^2]<0$$
rather than pointwise inequalities.

Recently, perhaps beginning with Savchuk and Shkalikov \cite{MR1756602}, there has been an increased interest in Sturm-Liouville equations with distributional coefficients.
Eckhardt et al. \cite{MR3046408} pointed out that all these situations (and more) are covered by the equation
\begin{equation}\label{de}
-(p(u'+su))'+rp(u'+su)+qu=0
\end{equation}
where $1/p$, $r$, $s$, and $q$ are real and integrable\footnote{In the case of $1/p$ we really mean here and below that $p$ is real-valued and $1/p$ is integrable.}.
Our goal here is to obtain a generalization of Leighton's comparison theorem covering two equations of the form \eqref{de}.

Our main result is Theorem \ref{main} which, together with its proof, is contained in Section \ref{ct}.
In Section \ref{sce} we discuss several special cases and examples to illustrate the use of the main theorem.
We also provide an appendix where we gather some known results for the convenience of the reader.

There are many excellent books concerned with comparison theorems of which we only mention Swanson \cite{zbMATH03303390}.
Hinton \cite{zbMATH02247495} provides a survey of the subject's history.

Finally we note that all our integrals are Lebesgue integrals unless indicated otherwise.

\section{The comparison theorem} \label{ct}
Solutions of $-(pu')'+qu=0$ are continuously differentiable, if $q$ and $1/p$ are continuous.
It is then clear that the term $\tu/u$ appearing in Picone's identity has finite limits at $a$ and $b$ even if $u$ vanishes there (recall that $\tu$ does, too).
The case where $q$ and $1/p$ are merely integrable may be treated by changing the independent variable according to $x\mapsto t=\int_a^x 1/p$.

\begin{lem}
Suppose $1/p,q\in L^1((a,b),\bb R)$ and that $p>0$ almost everywhere.
If there is a non-trivial real-valued function $\phi$ which is absolutely continuous on $[a,b]$, vanishes at $a$ and $b$, and satisfies
\begin{equation}\label{qf1}
\int_a^b (p\phi^{\prime2}+q\phi^2)\leq0,
\end{equation}
then every real-valued solution of $-(pu')'+qu=0$ has a zero in $(a,b)$ unless it is a constant multiple of $\phi$.
The latter case cannot occur when the inequality in \eqref{qf1} is strict.
\end{lem}

We emphasize that $\phi$ need not be a solution of a differential equation.

\begin{proof}
Note that our hypothesis \eqref{qf1} implies that $p\phi^{\prime2}$ is integrable since $q\phi^2$ is.

Let $\psi$ be a real-valued solution of $-(pu')'+qu=0$ which does not vanish in $(a,b)$.
We may assume that $\psi>0$ on $(a,b)$.

Define $g=p\psi'\phi^2/\psi$ on $(a,b)$.
We claim that $g$ has limit $0$ at both $a$ and $b$ which implies that $g$ is absolutely continuous on $[a,b]$ and hence that $\int_a^b g'=0$.
Consider the behavior of $g$ near $a$.
Our claim is obvious when $\psi(a)\neq0$, so we assume $\psi(a)=0$.
The function $k$ defined by $k(x)=\int_{a}^{x}1/p$ is absolutely continuous and strictly increasing.
Since $k'(x)=0$ only on a set of measure $0$ it follows from Lemma \ref{acinv} that $k^{-1}$ is also absolutely continuous (and strictly increasing).
Hence, by Theorem \ref{eandu} and Lemma \ref{accomp}, $\psi_0=\psi\circ k^{-1}$ and $\psi_0'=(p\psi')\circ k^{-1}$ are absolutely continuous on $[0,k(b)]$.
Since $\psi$ is not the trivial solution, $\psi_0'(0)>0$ and hence $\psi_0'(t)\geq C$ for some $C>0$ at least when $t$ is in some neighborhood of $0$.
Therefore $\psi(x)\geq C k(x)$ if $x$ is sufficiently close to $a$.
Next, the Cauchy-Schwarz inequality applied to $\phi(x)=\int_a^x p^{-1/2}p^{1/2}\phi'$ gives
$\phi(x)^2\leq k(x) \int_{a}^{x}p\phi^{\prime2}$ for all $x \in (a,b)$.
Therefore, if $x$ is sufficiently close to $a$,
$$0\leq\frac{\phi(x)^2}{\psi(x)}\leq\frac{1}{C} \int_{a}^{x}p\phi^{\prime2}$$
which tends to $0$ as $x$ tends to $a$.
Since a similar argument works at $b$ the proof of our claim is complete.

Now $g'=p\phi^{\prime2}+q\phi^2-p\psi^2(\phi/\psi)^{\prime2}$ (this is a variant of Picone's identity).
Hence
$$0\leq\int_a^b p\psi^2(\phi/\psi)^{\prime2}=\int_a^b (p\phi^{\prime2}+q\phi^2)\leq0$$
and so $\phi/\psi$ must be constant. Since $\phi$ is not trivial, this constant cannot be zero thus proving the lemma.
\end{proof}

We now extend the previous lemma to cover the general equation \eqref{de}.
We denote the antiderivatives of $s$ and $r$ which vanish at $a$ by $S$ and $R$, respectively.
\begin{lem} \label{L2}
Suppose $1/p,q,r,s\in L^1((a,b),\bb R)$ and that $p>0$ almost everywhere.
If there is a non-trivial real-valued function $\phi$ which is absolutely continuous on $[a,b]$, vanishes at $a$ and $b$, and satisfies
\begin{equation}\label{qf2}
\int_a^b \e^{S-R}(p(\phi'+s\phi)^2+q\phi^2) \leq0,
\end{equation}
then every real-valued solution of \eqref{de} has a zero in $(a,b)$ unless it is a constant multiple of $\phi$.
The latter case cannot occur when the inequality in \eqref{qf2} is strict.
\end{lem}

\begin{proof}
Assume $\psi$ solves equation \eqref{de} and that it is positive on $(a,b)$.
Define $p_0=p\e^{-S-R}$, $q_0=q\e^{-S-R}$, $\phi_0=\phi\e^{S}$ and $\psi_0=\psi\e^{S}$.
Then
$$\int_a^b \e^{S-R}(p(\phi'+s\phi)^2+q\phi^2)=\int_a^b (p_0\phi_0^{\prime2}+q_0\phi_0^2)$$
and
$$-(p(\psi'+s\psi))'+r p(\psi'+s\psi)+q\psi=[-(p_0\psi_0')'+q_0\psi_0]\e^{R}.$$
Since $\psi_0>0$ on $(a,b)$ the previous lemma shows that $\psi_0$ is a constant multiple of $\phi_0$ and hence $\psi$ a constant multiple of $\phi$.
\end{proof}

Now the question arises of how to find a function $\phi$ which satisfies \eqref{qf2}.
The idea of a comparison theorem is to look for it among the solutions of a related (but better known) equation with coefficients $(\tp,\tq,\tr,\ts)$.
In fact, we will generalize this idea by multiplying such a solution with a positive absolutely continuous function.
Any such function can be written as $\e^F$ where $F$ is absolutely continuous and real. We denote $F'$ by $f$.
Thus we set $\phi=\e^{F}\tu$ where $\tu$ satisfies
\begin{equation}\label{de1}
-(\tp(\tu'+\ts\tu))'+\tr\tp(\tu'+\ts\tu)+\tq\tu=0
\end{equation}
and $\tu(a)=\tu(b)=0$.
Condition \eqref{qf2} becomes then
\begin{equation}\label{precon}
\int_a^b e^{2F+S-R}(p(\tu'+(f+s)\tu)^2+q\tu^2)\leq0.
\end{equation}
Now multiply equation \eqref{de1} by $\e^G\tu $ where $G$ is real and absolutely continuous.
Let $g=G'$ and note that it is integrable.
Then we get after an integration by parts
\begin{multline*}
0=\int_a^b \e^G\tu [-(\tp (\tu '+\ts \tu ))'+\tr \tp (\tu '+\ts \tu )+\tq \tu ]\\
 =\int_a^b \e^G[(\tu '+g\tu ) \tp (\tu '+\ts \tu )+\tr \tp (\tu '+\ts \tu )\tu +\tq \tu ^2]\\
 =\int_a^b \e^G[\tp (\tu '+\ts \tu )^2+\tp (g+\tr -\ts )(\tu '+\ts \tu )\tu +\tq \tu ^2].
\end{multline*}
Subtracting this from \eqref{precon} we obtain the condition
\begin{equation}\label{con}
\int_a^b[A(\tu '+\ts \tu )^2+B(\tu '+\ts \tu )\tu +C\tu ^2]\leq0
\end{equation}
where
$$A=p\e^{2F+S-R}-\tp \e^G,$$
$$B=2p(f+s-\ts )\e^{2F+S-R}-\tp (g+\tr -\ts )\e^G,$$
and
$$C=(q+ p(f+s-\ts )^2)\e^{2F+S-R}-\tq \e^G.$$
In \eqref{con} we tacitly assume the integrability of the integrand.
A sufficient condition for this is the integrability of $A/\tp ^2$, $B/\tp $, and $C$.

The following result is now an immediate corollary of Lemma \ref{L2}.
\begin{thm}[The Comparison Theorem] \label{main}
Suppose $1/p,1/\tp ,q,\tq ,r,\tr ,s,\ts$ are all in $L^1((a,b),\bb R)$, that $p$ and $\tp $ are positive almost everywhere, and that the differential equation
$$-(\tp (u'+\ts u))'+\tr \tp (u'+\ts u)+\tq u=0$$
has a non-trivial real solution $\tu $ which vanishes at $a$ and $b$ and satisfies the inequality \eqref{con} for some choice of real absolutely continuous functions $F$ and $G$.
Then every real solution of
$$-(p(u'+su))'+rp(u'+su)+qu=0$$
has a zero in $(a,b)$ unless it is a constant multiple of $\tu\e^{F}$.
The latter case cannot occur when the inequality in \eqref{con} is strict.
\end{thm}

\section{Special cases and examples} \label{sce}
\subsection{The generalized Sturm-Picone theorem}
In analogy to $S$ and $R$ we will also use the antiderivative $\tilde S$ of $\ts$ which vanishes at $a$.
\begin{thm}
Suppose $1/p, 1/\tp, q, \tq ,s, \ts\in L^1((a,b),\bb R)$, that $0<p\leq \tp $ and $q\leq \tq $ almost everywhere, and that
$$\mu=\tp(s-\ts)\e^{-2S}$$
is non-decreasing on $[a,b]$.
If the differential equation
$$-(\tp (u'+\ts u))'+\ts \tp (u'+\ts u)+\tq u=0$$
has a non-trivial real solution $\tu $ with zeros at $a$ and $b$, then every real solution of
$$-(p(u'+su))'+sp(u'+su)+qu=0$$
has a zero in $(a,b)$ unless it is a constant multiple of $\tu\e^{\tilde S-S}$.
The latter case cannot occur when one of the inequalities $p\leq \tp $ or $q\leq \tq$ is strict on a set of positive measure or if $\mu$ is strictly increasing on some subinterval of $(a,b)$.
\end{thm}

\begin{proof}
Choose $G=2F=2\tilde S-2S$ and set $r=s$ and $\tr=\ts$ in inequality \eqref{con}.
Then $A/\tp ^2$, $B/\tp $, and $C$ are integrable and we have $A=(p-\tp)\e^G\leq0$ and $C=(q-\tq)\e^G\leq0$.
Setting $\tv=\tu\e^{\tilde S}$ gives
$$B(\tu '+\ts \tu )\tu=\mu(\tv^2)'.$$
Using Theorem 3.36 of Folland \cite{MR1681462} we get therefore
$$\int_a^b B(\tu '+\ts \tu )\tu= -\int_{[a,b)} \tv^2 d\mu\leq0.$$
Now apply Theorem \ref{main}.
\end{proof}

We chose $r=s$ and $\tr=\ts$ only for simplicity. A similar result holds also in the general case.
More importantly, perhaps, it is not necessary to have $\mu$ finite at $a$ and $b$.
It suffices to assume that $\mu$ is non-decreasing on $(a,b)$ and to require
$$\lim_{x\downarrow a}\mu(x)\tu(x)^2=\lim_{x\uparrow b}\mu(x)\tu(x)^2=0.$$

\subsection{The generalized Sturm separation theorem}
The following is a slight generalization of Theorem 11.1 in Eckhardt et al. \cite{MR3046408} (who have $r=s$).
\begin{thm}
Suppose $1/p,q,r,s\in L^1((a,b),\bb R)$, that $p>0$ almost everywhere, and that the differential equation
$$-(p(u'+su))'+rp(u'+su)+qu=0$$
has real solutions $u$ and $\tu$.
If $\tu$ is non-trivial but has zeros at $a$ and $b$, then $u$ has a zero in $(a,b)$ unless it is a constant multiple of $\tu$.
\end{thm}
\begin{proof}
Choosing $F=0$, $G=S-R$, $\tp=p$, $\tq=q$, $\tr=r$, and $\ts=s$ gives $A=B=C=0$ in \eqref{con}.
\end{proof}

\subsection{Distributional potentials}
Here we consider a Schr\"odinger equation with a distributional potential\footnote{Appendix \ref{AppB} gathers some basic facts about distributions.} to obtain a (slight) generalization of Theorem 2 of Ben Amara and Shkalikov \cite{MR2664498} or
Theorem 4.1 of Homa and Hryniv \cite{MR3162802}.
Note that, for our approach, incorporating a coefficient $p$ causes only a minor inconvenience at least when $p$ and $1/p$ are bounded.

If $u$ is in the Sobolev space $W^{1,2}((a,b))$ (i.e., $u$ is absolutely continuous on $[a,b]$ and $u'\in L^{2}((a,b))$), then $u''$ is a distribution in $W^{-1,2}((a,b))$ defined by $u''(\phi)=-\int_a^b u'\phi'$.
Moreover, if $u\in W^{1,2}((a,b))$ and if $v\in W^{-1,2}((a,b))$ has antiderivative $V\in L^{2}((a,b))$, then we may define $(vu)(\phi)=-\int_a^b V(u\phi)'$ which shows that $vu$ is also in $W^{-1,2}((a,b))$.
We may therefore pose the differential equation
$$-u''+vu=0.$$
Thus $u\in W^{1,2}((a,b))$ is a solution of this equation, if, for all test functions $\phi$,
$$0=(-u''+vu)(\phi)=\int_a^b ((u'-Vu)\phi'-Vu'\phi)=\int_a^b (u'-Vu+W)\phi'$$
where $W$ is an antiderivative of the integrable function $Vu'$ and hence absolutely continuous.
By Du Bois-Reymond's lemma \ref{dBR} $u'-Vu+W$ is constant which implies that $u'-Vu$ is absolutely continuous and hence $(u'-Vu)'=-W'=-Vu'=-V(u'-Vu)-V^2u$ almost everywhere.
It follows that $u$ satisfies equation \eqref{de} with $p=1$, $q=-V^2$ and $r=s=-V$.
Conversely, since $s=-V\in L^{2}((a,b))$, a solution of \eqref{de} is necessarily in $W^{1,2}((a,b))$ and solves $-u''+vu=0$ in the sense of distributions.

\begin{thm}
Suppose $v$ and $\tv$ are real distributions in $W^{-1,2}((a,b))$, that $\tv-v$ is a non-negative measure, and that the differential equation  $-u''+\tv u=0$ has a non-trivial real solution $\tu$ with zeros at $a$ and $b$.
Then every real solution of $-u''+vu=0$ has a zero in $(a,b)$ unless it is a constant multiple of $\tu$.
\end{thm}

\begin{proof}
Since $\tv-v$ is a non-negative measure we have that $\mu=\tV-V$ is non-decreasing.
With $F=G=0$ we have $A=B\tilde s+C=0$ in inequality \eqref{con} which then becomes
$$\int_a^b 2(\tV-V)\tu'\tu=\int_a^b \mu(\tu^2)'=-\int_{[a,b)} \tu^2 d\mu\leq0$$
using again Theorem 3.36 of Folland \cite{MR1681462}.
\end{proof}

\subsection{Difference equations}
A comparison theorem for the Jacobi difference equation is known at least since the work of Fort \cite{MR0024567} in 1948.
However, it may be viewed as a special case of Theorem \ref{main} as we will show now.

Let $\alpha$ be a sequence of positive numbers defined on $\bb N_0$ and $\beta$ a sequence of real numbers defined on $\bb N$.
We consider the difference equation
\begin{equation}\label{dde}
\alpha_{n-1}u_{n-1}+\beta_nu_n+\alpha_nu_{n+1}=0, \quad N_0+1\leq n\leq N_1-1
\end{equation}
on a bounded interval\footnote{All closed intervals in this section are to be viewed as subsets of $\bb N_0$ while half open intervals are considered subsets of $\bb R$.} $[N_0,N_1]$ of $\bb N_0$ for which $N_1-N_0\geq2$.
One might want to write equation \eqref{dde} in terms of forward differences $u_{n+1}-u_n$.
It then reads
$$-\alpha_n(u_{n+1}-u_n)+\alpha_{n-1}(u_{n}-u_{n-1})+v_nu_n=0, \quad N_0+1\leq n\leq N_1-1$$
where $v_n=-\beta_n-\alpha_n-\alpha_{n-1}$.

A solution $u:[N_0,N_1]\to\bb R$ of \eqref{dde} may change sign without ever being zero.
We will therefore be interested in sign changes rather than zeros of solutions.
To be precise we will make the following definition.
\begin{defn}
The sequence $u:[N_0,N_1]\to\bb R$ changes sign on $[N_0,N_1]$ if there are $n,m\in[N_0,N_1]$ such that $u_nu_m<0$.
\end{defn}
If $u$ is a real non-trivial solution of \eqref{dde} and $u_n=0$ for some $n\in[N_0+1,N_1-1]$,
then $u_{n-1}$ and $u_{n+1}$ must have different signs.
In other words, if $u$ has a zero in $[N_0+1,N_1-1]$, then it changes sign on $[N_0,N_1]$.

\newcommand{\Alpha}{\tilde\alpha}
\begin{thm}
Suppose $\alpha$ and $\Alpha$ are positive sequences defined on a bounded interval $[N_0,N_1-1]$ and that $v$ and $\tv$ are real sequences on $[N_0+1,N_1-1]$.
If the difference equation
\begin{equation}\label{ddet}
-\Alpha_n(u_{n+1}-u_n)+\Alpha_{n-1}(u_{n}-u_{n-1})+\tv_nu_n=0, \quad N_0+1\leq n\leq N_1-1
\end{equation}
has a non-trivial solution $\tu:[N_0,N_1]\to\bb R$ such that $\tu_{N_0}=0$, $\tu_{N_1-1}\tu_{N_1}\leq0$, and
\begin{multline}\label{dcon}
\sum_{n=N_0+1}^{N_1-1} \big[(\alpha_{n-1}-\Alpha_{n-1})(\tu_{n}-\tu_{n-1})^2+(v_n-\tv_n) \tu_n^2\big] \\ +(\alpha_{N_1-1}-\Alpha_{N_1-1})(\tu_{N_1-1}^2-\tu_{N_1-1}\tu_{N_1})\leq0
\end{multline}
then every real solution of \eqref{dde} changes sign on $[N_0,N_1]$ unless it is a constant multiple of $\tu$.
The latter case cannot occur when the inequality in \eqref{dcon} is strict.
\end{thm}

\begin{proof}
We define each of $\tp$, $\tq$, and $\ts=\tr$ on the real interval $[N_0,N_1)$ to be piecewise constant,
specifically $\tp=\Alpha_n$ and $\ts=-\sum_{k=N_0+1}^n \tv_k/\Alpha_n$ on $[n,n+1)$ (in particular, $\ts=0$ on $[N_0,N_0+1)$) and $\tq=-\tp\ts^2$.
Moreover, we define $\tu$ to be continuous and piecewise linear assuming the given values at the points of $[N_0,N_1]$.
Thus, on $[n,n+1)$ we have $\tilde u(x)=\tu_n+(\tu_{n+1}-\tu_n)(x-n)$.
Since $n\mapsto\tu_n$ satisfies the difference equation \eqref{ddet} the function $x\mapsto\tu(x)$ satisfies $-(\tp(\tu'+\ts\tu))'+\ts\tp(\tu'+\ts\tu)+\tq\tu=0$.
Analogously, we define $p$, $q$ and $s=r$ with the aid of the coefficients $\alpha_n$ and $v_n$.
Now set $a=N_0$ and $b$ to be the point where the straight line segment joining the points $(N_1-1,\tu_{N_1-1})$ and $(N_1,\tu_{N_1})$ crosses the abscissa.
Then, choosing $F=G=0$, the left-hand side of inequality \eqref{con} equals the left-hand side of inequality \eqref{dcon}.
Since the solutions of the differential equation \eqref{de} and the difference equation \eqref{dde} are in one-to-one correspondence we may apply Theorem \ref{main} to obtain our conclusion.
\end{proof}

\subsection{Leighton's example}
Leighton \cite{MR0140759} discusses the following example to illustrate the use of his integral condition as compared to Picone's pointwise condition.
Here we show that a careful choice of $F$ and $G$ gives another improvement.

The function $\tu(x)=\sin(x)$ solves the equation $-u''-u=0$.
With its help we want to draw conclusions about the zeros of the solutions of $-u''+qu=0$ when $q=k-1-x$.
Using this in \eqref{con} with the choice $2F=G$ we get $A=B=0$ and $C(x)=(k-x+g(x)^2/4)\e^{G(x)}$.
With Leighton's choice $G=0$ we obtain $\int_0^\pi C(x) \sin(x)^2 dx=\pi(2k-\pi)/4$.
If $k>\pi/2$ this is positive and does not allow a conclusion.
With $G=0.6x$ and $k\leq1.672$ we obtain that $\int_0^\pi C(x) \sin(x)^2 dx$ is negative.
In this case every solution of $-u''+qu=0$ has a zero in $(0,\pi)$.
Note that for $k\geq1.676$ one can find solutions without zeros in $(0,\pi)$.

\appendix
\section{Existence and uniqueness}
The following existence and uniqueness theorem is an important ingredient of our results.
\begin{thm} \label{eandu}
Suppose $1/p$, $q$, $r$ and $s$ are real-valued and integrable on a bounded interval $(a,b)$, $x_0\in[a,b]$, and $A,B\in\bb R$.
Then there is a unique real-valued solution $u$ of the differential equation
$$-(p(u'+su))'+rp(u'+su)+qu=0$$
such that $u$ and $p(u'+su)$ are absolutely continuous on $[a,b]$ and satisfy the initial conditions $u(x_0)=A$ and $[p(u'+su)](x_0)=B$.
\end{thm}

Actually $p(u'+su)$ is, in general, only almost everywhere equal to an absolutely continuous function.
Similarly the equality in the differential equation may only hold almost everywhere.

The theorem follows from a standard iteration scheme since the equation is equivalent to the system $U'=MU$
where $u$ is the first component of $U$, and
$$M=\begin{pmatrix}-s&1/p \\ q&r  \end{pmatrix}$$
is integrable.

\section{Absolutely continuous functions}
Since we need one or two facts about absolutely continuous functions whose proofs do not seem readily available we provide this appendix.

While sums and products of absolutely continuous functions are again absolutely continuous, the same cannot necessarily be said for compositions.
Instead, the composition of two absolutely continuous functions is absolutely continuous if and only if it is of bounded variation (see Natanson \cite{MR0067952}, Theorem IX.3.5).
The following special cases suffice for our purposes and may be proved in a straightforward manner.

\begin{lem} \label{accomp}
Suppose $f:[\alpha,\beta]\to[a,b]$ and $g:[a,b]\to [A,B]$ are absolutely continuous, that $f$ is strictly increasing, and that $h:[A,B]\to\bb R$ is Lipschitz.
Then $g\circ f:[\alpha,\beta]\to[A,B]$ and $h\circ g:[a,b]\to\bb R$ are also absolutely continuous.
\end{lem}

Related to this, but apparently less well known, is the question of the absolute continuity of the inverse of an absolutely continuous function.
Spataru \cite{MR2112931} gave an example of a strictly increasing absolutely continuous function whose inverse is not absolutely continuous.
\begin{lem} \label{acinv}
If $f:[a,b]\to\bb R$ is strictly increasing and continuous the following statements hold.
\begin{enumerate}
  \item $f$ is absolutely continuous on $[a,b]$ if and only if the Lebesgue measure of $f(\{x\in[a,b]:f'(x)=\infty\})$ equals $0$.
  \item $f^{-1}$ is absolutely continuous on $[f(a),f(b)]$ if and only if the Lebesgue measure of $\{x\in[a,b]: f'(x)=0\}$ equals $0$.
\end{enumerate}
\end{lem}

These results are stated as an exercise in Natanson \cite{MR0067952}.
Their proofs require the notion of a derived number and an additional result (Lemma \ref{BBT} below), which in turn, relies on Vitali's covering theorem.
We denote Lebesgue measure by $\bm$ and the corresponding outer measure by $\bm^*$.

Let the function $f:[a,b]\to\bb R$ and the point $x_0\in[a,b]$ be given.
If $n\mapsto x_n\in[a,b]\setminus\{x_0\}$ converges to $x_0$, the sequence $n\mapsto (f(x_n)-f(x_0))/(x_n-x_0)$ has at least one limit point in $[-\infty,\infty]$.
Any such limit point is called a derived number for $f$ at $x_0$.
Clearly, $f$ is differentiable at $x_0$ (allowing $\pm\infty$ as derivatives), if all derived numbers coincide.

The following is Lemma 7.1 in Bruckner, Bruckner, and Thomson \cite{BBT}.
\begin{lem} \label{BBT}
Let $f:[a,b]\to\bb R$ be strictly increasing and let $E$ be a subset of $[a, b]$.
If at each point $x\in E$ there exists a derived number not exceeding $p$, then $\bm^*(f(E)) \leq p \bm^*(E)$.
\end{lem}

\begin{proof}[Proof of Lemma \ref{acinv}]
We will first prove the only if direction of (1), and then the if direction. Finally we will show that (1) and (2) are equivalent.
First note that $f$ is of bounded variation and thus has a finite derivative almost everywhere.
Since $f$ is increasing all derived numbers are non-negative.
Let $A$ be the set of those $x$ where $f'(x)$ exists and is finite, $B$ the set where $f'(x)=\infty$, and $C$ the set where no derivative exists (not even an infinite one).
These sets are pairwise disjoint and their union is $[a,b]$.
We know that $\bm(B)=\bm(C)=0$.

Now assume that $f$ is absolutely continuous.
Then it maps sets of measure $0$ to sets of measure $0$ and, in particular, $\bm(f(B))=0$, completing our first step.

For the second step $\bm(f(B))=0$ is the assumption and our main objective is to show that $\bm(f(C))=0$ but first we need to investigate whether images of measurable sets are measurable.
To this end let $\mc A=\{E:\text{$f(E)$ is Lebesgue measurable}\}$.
It is easy to see that $\mc A$ is a $\sigma$-algebra in $[a,b]$ which contains all relatively open subintervals of $[a,b]$ and hence all its Borel sets.
In particular, the image of any Borel set is Lebesgue measurable.
Thus $E\mapsto \mu(E)=\bm(f(E))$ is a measure defined on the Borel sets of $[a,b]$.
Since $\mu([a,x))=f(x)-f(a)$ we find that $\mu$ is the Lebesgue-Stieltjes measure generated by $f$ (or rather the restriction of this to the Borel sets).
Recall that $f$ is an absolutely continuous function if and only if $\mu$ is absolutely continuous with respect to Lebesgue measure.
By the Radon-Nikodym theorem $\mu=\mu_a+\mu_s$ where $\mu_a$ is absolutely continuous with respect to $\bm$ while $\mu_s$ and $\bm$ are mutually singular.
Thus we need to show that $\mu_s=0$.

Since $\mu_s$ and $\bm$ are mutually singular, there is a set $S\subset[a,b]$ such that $\bm(S)=0=\mu_s([a,b]\setminus S)$.
Since $\mu(B)=0$ and $\bm(C)=0$ we may assume that $B\subset S^c$ and $C\subset S$.
For $n\in\bb N$ let $S_n$ be the set of those $x\in S$ for which there is a derived number for $f$ smaller than $n$.
Then $S=\bigcup_{n=1}^\infty S_n$.
But, by Lemma \ref{BBT}, $\mu_s(S_n)=\mu(S_n)\leq n\bm(S_n)=0$ which implies $\mu_s=0$.
We have now proved the first statement.

To prove that the second statement is equivalent to the first let $g=f^{-1}$.
Then note that $g$ is strictly increasing and continuous and that $g(\{t\in[f(a),f(b)]:g'(t)=\infty\})=\{x\in[a,b]: f'(t)=0\}$.
\end{proof}

\section{Distributions} \label{AppB}
Distributions are linear functionals on a set of test functions\footnote{See Gelfand and Shilov \cite{MR0435831} or H\"ormander \cite{MR1996773} for details going beyond the very basics discussed here.}.
Specifically, in our context, test functions are complex-valued, compactly supported functions defined on $\iOmega$ which have derivatives of all orders.
The space of test functions is denoted by $\mc D(\iOmega)$.
A linear functional $u$ on $\mc D(\iOmega)$ is a distribution, if for every compact set $K\subset\iOmega$ there are constants $C>0$ and $k\in\bb N_0$ such that
\begin{equation}\label{eq:1.1.1}
|u(\phi)|\leq C \sum_{j=0}^k \sup\{|\phi^{(j)}(x)|:x\in K\}
\end{equation}
whenever the test function $\phi$ has its support in $K$.
The set of all distributions on $\iOmega$ is denoted by $\mc D'(\iOmega)$.
It becomes a linear space upon defining $\alpha u_1+\beta u_2$ by $(\alpha u_1+\beta u_2)(\phi)=\alpha u_1(\phi)+\beta u_2(\phi)$ whenever $u_1, u_2\in\mc D'(\iOmega)$ and $\alpha,\beta\in\bb C$.

If $u$ is a distribution then so is $\phi\mapsto -u(\phi')$.
This distribution is called the derivative of $u$ and is denoted by $u'$.
Distributions also have antiderivatives.
To see this fix $\psi\in \mc D(\iOmega)$ with $\int_a^b\psi=1$ so that $\varphi(x)=\int_a^x (\phi-\psi\int_a^b\phi)$ defines a test function in $\mc D(\iOmega)$ whenever $\phi$ is one.
Now, if $u$ is a distribution, define the linear functional $v:\phi\mapsto -u(\varphi)$. It is easy to check that $v$ is a distribution.
Also, since $\int_a^b \phi'=0$, we find that $v'(\phi)=-v(\phi')=u(\phi)$, i.e., $v$ is an antiderivative of $u$.
Two antiderivatives of a distribution differ by only a constant as the following lemma shows.

\begin{lem}[Du Bois-Reymond] \label{dBR}
Suppose the derivative of the distribution $v$ is zero.
Then $v$ is the constant distribution, i.e., there is a constant $C$ such that $v(\phi)=C\int_a^b\phi$ for all $\phi\in\mc D(\iOmega)$.
\end{lem}

A distribution $u$ is called real if $u(\phi)$ is real whenever $\phi$ assumes only real values.
It is called non-negative if $u(\phi)\geq0$ whenever $\phi\geq0$.

We conclude by giving some pertinent examples of distributions.
A complex measure $\mu$ on $[a,b]$ may be identified with the distribution $\phi\mapsto \int_{[a,b]} \phi\ d\mu$.
In particular, if $\mu$ is absolutely continuous with respect to Lebesgue measure $\bm$ and if $f\in L^1(\iOmega)$ is the corresponding Radon-Nikodym derivative, we have the distribution $\phi\mapsto \int_a^b f\phi\ d\textrm{m}$.
The antiderivative of $\phi\mapsto \int_{[a,b]} \phi\ d\mu$ is given by $\phi\mapsto \int_{[a,b]} F\phi\ d\textrm{m}$ where $F(x)=\mu([a,x))$ is a function of bounded variation.

The class $W^{-1,2}(\iOmega)$ consists of the distributions defined by $\phi\mapsto -\int_a^b F \phi'$ where $F\in L^2(\iOmega)$.
Note that the complex measures are a subset of $W^{-1,2}(\iOmega)$.

\bibliographystyle{plain}

\begin{thebibliography}{10}

\bibitem{BBT}
Andrew~M. Bruckner, Judith~B. Bruckner, and Brian~S. Thomson.
\newblock {\em Real Analysis}.
\newblock ClassicalRealAnalysis.com, 2008.
\newblock Second edition.

\bibitem{MR3046408}
Jonathan Eckhardt, Fritz Gesztesy, Roger Nichols, and Gerald Teschl.
\newblock Weyl-{T}itchmarsh theory for {S}turm-{L}iouville operators with
  distributional potentials.
\newblock {\em Opuscula Math.}, 33(3):467--563, 2013.

\bibitem{MR1681462}
Gerald~B. Folland.
\newblock {\em Real analysis}.
\newblock Pure and Applied Mathematics (New York). John Wiley \& Sons, Inc.,
  New York, second edition, 1999.
\newblock Modern techniques and their applications, A Wiley-Interscience
  Publication.

\bibitem{MR0024567}
Tomlinson Fort.
\newblock {\em Finite {D}ifferences and {D}ifference {E}quations in the {R}eal
  {D}omain}.
\newblock Oxford, at the Clarendon Press, 1948.

\bibitem{MR0435831}
I.~M. Gel{\cprime}fand and G.~E. Shilov.
\newblock {\em Generalized functions. {V}ol. 1}.
\newblock Academic Press [Harcourt Brace Jovanovich Publishers], New York, 1964
  [1977].
\newblock Properties and operations, Translated from the Russian by Eugene
  Saletan.

\bibitem{zbMATH02247495}
Don {Hinton}.
\newblock {Sturm's 1836 oscillation results evolution of the theory.}
\newblock In {\em {Sturm-Liouville theory. Past and present. Selected survey
  articles based on lectures presented at a colloquium and workshop in Geneva,
  Italy, September 15--19, 2003 to commemorate the 200th anniversary of the
  birth of Charles Fran\c cois Sturm}}, pages 1--27. Basel: Birkh\"auser, 2005.

\bibitem{MR3162802}
Monika Homa and Rostyslav Hryniv.
\newblock Comparison and oscillation theorems for singular {S}turm-{L}iouville
  operators.
\newblock {\em Opuscula Math.}, 34(1):97--113, 2014.

\bibitem{MR1996773}
Lars H{\"o}rmander.
\newblock {\em The analysis of linear partial differential operators. {I}}.
\newblock Classics in Mathematics. Springer-Verlag, Berlin, 2003.
\newblock Distribution theory and Fourier analysis, Reprint of the second
  (1990) edition [Springer, Berlin; MR1065993 (91m:35001a)].

\bibitem{MR0140759}
Walter Leighton.
\newblock Comparison theorems for linear differential equations of second
  order.
\newblock {\em Proc. Amer. Math. Soc.}, 13:603--610, 1962.

\bibitem{MR0067952}
I.~P. Natanson.
\newblock {\em Theory of functions of a real variable}.
\newblock Frederick Ungar Publishing Co., New York, 1955.
\newblock Translated by Leo F. Boron with the collaboration of Edwin Hewitt.

\bibitem{Picone1910}
Mauro Picone.
\newblock Sui valori eccezionali di un parametro da cui dipende un'equazione
  differenziale lineare ordinaria del second'ordine.
\newblock {\em Annali della Scuola Normale Superiore di Pisa - Classe di
  Scienze}, 11:1--144, 1910.

\bibitem{MR1756602}
A.~M. Savchuk and A.~A. Shkalikov.
\newblock Sturm-{L}iouville operators with singular potentials.
\newblock {\em Mathematical Notes}, 66(6):741--753, 1999.
\newblock Translated from Mat. Zametki, Vol. 66, pp. 897--912 (1999).

\bibitem{MR2664498}
A.~A. Shkalikov and Zh. Ben~Amara.
\newblock Oscillation theorems for {S}turm-{L}iouville problems with
  distribution potentials.
\newblock {\em Vestnik Moskov. Univ. Ser. I Mat. Mekh.}, (3):43--49, 2009.

\bibitem{MR2112931}
Silvia Sp{\u{a}}taru.
\newblock An absolutely continuous function whose inverse function is not
  absolutely continuous.
\newblock {\em Note Mat.}, 23(1):47--49, 2004/05.

\bibitem{Sturm1836}
Sturm.
\newblock Mémoire sur les équations différentielles linéaires du second ordre.
\newblock {\em Journal de Mathématiques Pures et Appliquées}, 1:106--186, 1836.

\bibitem{zbMATH03303390}
C.A. {Swanson}.
\newblock {\em {Comparison and oscillation theory of linear differential
  equations}}.
\newblock Academic Press, 1968.

\end{thebibliography}
\def\cprime{$'$} \def\soft#1{\leavevmode\setbox0=\hbox{h}\dimen7=\ht0\advance
  \dimen7 by-1ex\relax\if t#1\relax\rlap{\raise.6\dimen7
  \hbox{\kern.3ex\char'47}}#1\relax\else\if T#1\relax
  \rlap{\raise.5\dimen7\hbox{\kern1.3ex\char'47}}#1\relax \else\if
  d#1\relax\rlap{\raise.5\dimen7\hbox{\kern.9ex \char'47}}#1\relax\else\if
  D#1\relax\rlap{\raise.5\dimen7 \hbox{\kern1.4ex\char'47}}#1\relax\else\if
  l#1\relax \rlap{\raise.5\dimen7\hbox{\kern.4ex\char'47}}#1\relax \else\if
  L#1\relax\rlap{\raise.5\dimen7\hbox{\kern.7ex
  \char'47}}#1\relax\else\message{accent \string\soft \space #1 not
  defined!}#1\relax\fi\fi\fi\fi\fi\fi}

\end{document}